\documentclass[11pt, oneside]{article}   	
\usepackage{geometry}                		
\usepackage{graphicx}	

\usepackage{amssymb}
\usepackage{amsmath}
\usepackage{amsthm}

\usepackage[english]{babel} 

\usepackage{mathrsfs}

\usepackage{setspace}
\usepackage{changepage}

\usepackage{xcolor}

\usepackage{tikz}
\usetikzlibrary{matrix,arrows,chains,positioning,scopes}
\tikzset{>=stealth}
\usepackage{tikz-cd}

\usepackage[labelfont=bf,textfont=it]{caption}

\newcommand{\R}{\mathbb{R}}

\DeclareMathOperator{\CAT}{CAT}

\DeclareMathOperator{\co}{co}
\DeclareMathOperator{\cco}{\overline{co}}
\DeclareMathOperator{\diam}{diam}

\newtheorem{thm}{Theorem}[section]
\newtheorem{prop}[thm]{Proposition}
\newtheorem{lem}[thm]{Lemma} 
\newtheorem{cor}[thm]{Corollary} 

\theoremstyle{definition}
\newtheorem{defn}[thm]{Definition}

\newtheorem{ex}[thm]{Example}
\newtheorem*{rem}{Remark} 
\newtheorem*{ack}{Acknowledgements}

\title{Geodesic bicombings on some hyperspaces}
\author{Logan S. Fox\footnote{Fariborz Maseeh Dept. of Math. and Stat., Portland State University; logfox@pdx.edu}} 
\date{March 23, 2022} 

\begin{document}

\maketitle

\begin{abstract} 
We show that if \((X,d)\) is a metric space which admits a consistent convex geodesic bicombing, then we can construct a conical bicombing on \(CB(X)\), the hyperspace of nonempty, closed, bounded, and convex subsets of \(X\) (with the Hausdorff metric). If \(X\) is a normed space or an \(\R\)-tree, this same method produces a consistent convex bicombing on \(CB(X)\). 
We follow this by examining a geodesic bicombing on the nonempty compact subsets of \(X\), assuming \(X\) is a proper metric space.
\end{abstract}

\section{Introduction} 

A \emph{hyperspace} is a topological space whose elements are subsets of some other topological space (called the \emph{base space}). Perhaps the most common example of a hyperspace is to take a metric space as the base space and form the (metric) hyperspace of nonempty compact subsets, together with the Hausdorff metric. What makes this example particularly interesting is that the hyperspace of nonempty compact subsets inherits many of the topological properties from the base space \cite{michael}. Our intention is to examine whether or not a hyperspace can inherit any \emph{geometric} properties from the base space. Some progress in this direction was made in \cite{sosov}, where it was shown that the Hausdorff metric is intrinsic if and only if the metric of the base space is intrinsic. We take this further by constructing \emph{geodesic bicombings} for some hyperspaces. 

\begin{defn}\label{defn: bicombing}
Let \((X,d)\) be a metric space. A \emph{geodesic bicombing} on \(X\) is a function \(\sigma : X \times X \times [0,1] \to X\) 
such that for any \(x,y\in X\) and \(s,t\in [0,1]\); \(\sigma(x,y,0) = x\), \(\sigma(x,y,1) = y\), and 
\[ d ( \sigma(x,y,s) , \sigma(x,y,t) ) = |t-s| d(x,y) . \] 
In other words, \(\sigma(x,y,\cdot)\) is a linearly reparametrized geodesic emanating from \(x\) and terminating at \(y\). 
The bicombing \(\sigma\) is \emph{conical} if for any \(a,b,c,d\in X\) and \(t\in [0,1]\), 
\[ d( \sigma(a,b,t) , \sigma(c,d,t) ) \leq (1-t) d(a,c) + t d(b,d) . \] 
We say that \(\sigma\) is \emph{convex} if for all \(a,b,c,d\in X\), the map \(t \mapsto d\big( \sigma(a,b,t) , \sigma(c,d,t) \big)\) is convex on \([0,1]\). 
Finally, we say that \(\sigma\) is \emph{consistent} if for every \(x,y\in X\) and \(r,s,t\in [0,1]\) with \(r<s\), 
\[ \sigma(\sigma(x,y,r),\sigma(x,y,s) , t) = \sigma(x,y, (1-t)r + t s ) . \] 
\end{defn}

See \cite{descombes} or \cite{DL2015} for a further discussion on the different kinds of bicombings and their relationship to spaces of nonpositive curvature. We note that every convex bicombing is conical, and every consistent conical bicombing is convex. 
Furthermore, every normed space (and every convex subset of a normed space) admits a consistent convex bicombing defined by 
\begin{equation}\label{eq: normed bicombing} 
\sigma(x,y,t) = (1-t)x + ty . 
\end{equation} 

Our main motivation comes from the following observation for normed spaces. 
Let \(X\) be a real normed space and let \(CB(X)\) be the collection of nonempty, closed, bounded, convex subsets of \(X\) (equipped with the Hausdorff metric). H\"ormander's embedding theorem (Theorem \ref{thm: hmndr}) tells us that \(CB(X)\) can be algebraically and isometrically embedded as a convex cone in a Banach space. It follows immediately that \(CB(X)\) admits a consistent convex bicombing. 
We generalize this by showing that if \(X\) is a metric space which admits a consistent convex bicombing, then we can construct a conical bicombing on \(CB(X)\) (Theorem \ref{thm: conical}).

In Section \ref{sec: hyperspaces} we establish some necessary results concerning the Hausdorff metric. In Section \ref{sec: bicombings cb(x)} we use H\"ormander's embedding theorem to define a consistent convex bicombing on \(CB(X)\), assuming \(X\) is a normed space. This method is then generalized to construct a conical bicombing on \(CB(X)\) for any \(X\) that admits a consistent convex bicombing. In Section \ref{sec: cases}, we give two special cases where the bicombing on \(CB(X)\) is consistent. Finally, in Section \ref{sec: bicombing k(x)} we define a geodesic bicombing for the nonempty compact subsets of a proper geodesic space. However, we show that this final bicombing is not necessarily conical even if the base space admits a consistent convex bicombing.

\section{Hyperspaces with the Hausdorff Metric}\label{sec: hyperspaces} 

Let \((X,d)\) be a metric space. We quickly recall that the distance from a point \(x\in X\) to a nonempty set \(A\subseteq X\) is \(d(x,A) = \inf_{a\in A} d(x,a)\) and the \emph{closed epsilon neighborhood} of a set \(A\) is 
\[ N_\varepsilon(A) = \{ x \in X : d(x,A) \leq \varepsilon \} . \] 
We now define the \emph{Hausdorff distance} between two nonempty sets \(A,B\subseteq X\) as 
\[ d_H(A,B) = \inf \{ \varepsilon >0 : A\subseteq N_{\varepsilon}(B) \text{ and } B \subseteq N_{\varepsilon}(A) \} . \] 
Given an arbitrary collection of subsets, the Hausdorff distance does not necessarily satisfy the axioms of a metric; however, it is a metric on the collections 
\begin{align*} 
B(X) & = \{ A\subseteq X : A \text{ is nonempty, closed, and bounded} \} \\ 
\text{and } K(X) & = \{ A\subseteq X : A \text{ is nonempty and compact} \} . 
\end{align*} 
If we further assume that \(X\) is a geodesic space, then we can consider subcollections of the above by requiring the sets in the collection be convex. 

\begin{defn}
Let \(X\) be a metric space with a geodesic bicombing \(\sigma\). A set \(C\subseteq X\) is \emph{convex}\footnote{Other authors may call this \(\sigma\)-convex, as it does in fact depend on the bicombing \(\sigma\).} if \(\sigma(x,y,t) \in C\) for every \(x,y\in C\) and \(t\in [0,1]\). Given a set \(A\subseteq X\), the \emph{convex hull} of \(A\), denoted by \(\co(A)\), is the intersection of all convex sets containing \(A\); and the \emph{closed convex hull} of \(A\), denoted by \(\cco(A)\), is the intersection of all closed convex sets containing \(A\). 
\end{defn} 

With this definition of convex sets in mind, whenever \(X\) is a geodesic space, we can consider the collections 
\begin{align*} 
CB(X) & = \{ A\in B(X) : A \text{ is convex} \} \\ 
\text{and } CK(X) & = \{ A \in K(X) : A \text{ is convex} \} . 
\end{align*}
When working in a normed space, we will always assume it to be equipped with the consistent convex bicombing defined by \eqref{eq: normed bicombing}. In fact, in many Banach spaces this is the unique consistent convex bicombing \cite{BM2019}. 
By assuming the linear bicombing of \eqref{eq: normed bicombing}, our definition of a convex set perfectly aligns with the vector space definition.

This leads us to our first result: If \(X\) admits a conical bicombing, then the convex hull of two sets is no further apart than the sets themselves.

\begin{lem}\label{lem: convex distance inequality}
If \(X\) admits a conical bicombing \(\sigma\), then for any \(A,B\subseteq X\), 
\[ d_H(\cco(A),\cco(B)) \leq d_H(A,B) . \] 
\end{lem}

\begin{proof}
Let \(A,B\subseteq X\) be given. If \(d_H(A,B)\) is infinite, then the statement is trivial, so suppose that \(d_H(A,B)\) is finite. 
Let \(\varepsilon>0\) be such that \(A\subseteq N_{\varepsilon}(B)\). 
Since \(B\subseteq \cco(B)\), we immediately have \(A\subseteq N_{\varepsilon}(\cco(B))\). It remains only to show that \(\cco(A)\subseteq N_\varepsilon(\cco(B))\). 

Fix \(x_1,x_2\in N_{\varepsilon}(\cco(B))\). For any arbitrary small \(\delta>0\), we can find \(b_1,b_2\in \cco(B)\) such that \(d(x_i,b_i) < d(x_i,\cco(B))+\delta\). 
Since \(\sigma\) is conical, for all \(t\in [0,1]\) we have 
\begin{align*} 
d\big(\sigma(x_1,x_2,t) , \sigma(b_1,b_2,t) \big) & \leq (1-t) d(x_1,b_1) + t d(x_2,b_2) 
\\ & < (1-t)d(x_1,\cco(B)) + td(x_2,\cco(B)) + \delta 
\\ & \leq \varepsilon + \delta . 
\end{align*}
Since \(\delta\) can be arbitrarily small, \(\sigma(x_1,x_2,t)\in N_\varepsilon(\cco(B))\) and so \(N_\varepsilon(\cco(B))\) is convex. Given that \(N_\varepsilon(\cco(B))\) is a closed convex set containing \(A\), we also have \(\cco(A) \subseteq N_\varepsilon(\cco(B))\). It now follows that 
\begin{align*} 
d_H(\cco(A),\cco(B)) & = \inf \big\{ \varepsilon >0 : \cco(A) \subseteq N_{\varepsilon}(\cco(B)) \ \mathrm{and} \ \cco(B) \subseteq N_{\varepsilon} (\cco(A)) \big\} 
\\ & \leq \inf \big\{ \varepsilon>0 : A \subseteq N_{\varepsilon}(B) \ \mathrm{and} \ B \subseteq N_\varepsilon(A)  \big\} 
\\ & = d_H(A,B) . \qedhere
\end{align*}
\end{proof}

\begin{rem} 
As a minor observation, the above lemma tells us that if \(X\) is equipped with a conical bicombing, then the map \(\cco : B(X) \to CB(X)\) is 1-Lipschitz. 
\end{rem}

We end this section with an equivalent formulation of the Hausdorff metric, which we will use throughout Sections \ref{sec: bicombings cb(x)} and \ref{sec: bicombing k(x)}. 

\begin{lem} 
The Hausdorff distance can be equivalently defined by 
\[ d_H(A,B) = \max \left\{ \sup_{a\in A} d(a,B) , \sup_{b\in B} d(b,A) \right\} . \] 
\end{lem}

\begin{proof}
First, note that \(x\in N_\varepsilon (A)\) if and only if \(d(x,A) \leq \varepsilon\). Therefore, 
\[\inf \{ \varepsilon > 0 : B\subseteq N_\varepsilon (A) \} = \sup_{b\in B} d(b,A) . \] 
It now follows that 
\begin{align*} 
d_H(A,B) & = \inf\{ \varepsilon>0 : A\subseteq N_\varepsilon (B) \text{ and } B\subseteq N_\varepsilon(A)\} 
\\ & = \max \big\{ \inf \{ \varepsilon > 0 : A\subseteq N_\varepsilon (B) \} , \inf \{ \varepsilon > 0 : B\subseteq N_\varepsilon (A) \} \big\} 
\\ & = \max \Big\{ \sup_{a\in A} d(a,B) , \sup_{b\in B} d(b,A) \Big\} . \qedhere
\end{align*}
\end{proof}

\section{A Conical Bicombing on \(CB(X)\)}\label{sec: bicombings cb(x)} 

In order to motivate our technique for constructing bicombings on \(CB(X)\), we will begin with the assumption that \(X\) is a real normed space (with dual space \(X^*\)). 
In this case, the bicombing follows immediately from H\"ormander's embedding theorem. 
To present the theorem, we define the \emph{support functional}, \(s_A : X^* \to \R\), of a nonempty set \(A\) by 
\[ s_A(x^*) = \sup_{a\in A} x^*(a) . \]

\begin{thm}[H\"ormander]\label{thm: hmndr} 
Let \(X\) be a real normed space. The map \(A\mapsto s_A\) is an algebraic and isometric embedding of \(CB(X)\) as a convex cone in the Banach space of bounded continuous functions on the closed unit ball of \(X^*\). 
\end{thm} 

For a proof of H\"ormander's theorem, see \cite[Theorem 3.2.9]{beer} or H\"ormander's original paper \cite{hormander}. 
Our interest lies in the following corollary. 



\begin{cor}\label{cor: hrmndr} 
If \(X\) is a real normed space, then 
\[ \Sigma (A,B,t) = \overline{(1-t)A + tB} \] 
defines a consistent convex bicombing on \((CB(X), d_H)\). 
\end{cor} 

We relegate the details of the proof of Corollary \ref{cor: hrmndr} to that of Proposition \ref{prop: hrmndr2}, where we show (perhaps unsurprisingly) this is also true when \(X\) is a complex normed space. 
For now, recall our assumption that every normed space is equipped with the bicombing \(\sigma(x,y,t) = (1-t)x+ty\). 
This allows us to equivalently define the bicombing \(\Sigma\) in Corollary \ref{cor: hrmndr} as 
\[ \Sigma (A,B,t) = \overline{\bigcup_{a\in A} \bigcup_{b\in B} \sigma(a,b,t)} . \]
It turns out that this idea generalizes to spaces with a consistent convex bicombing, except we need to take the closed convex hull.

\begin{lem}\label{lem: geodesic} 
If \(X\) admits a consistent convex bicombing \(\sigma\), then for any \(A,B\in CB(X)\), the map 
\[ \gamma(t) = \cco \bigg( \bigcup_{a\in A} \bigcup_{b\in B} \sigma(a,b,t) \bigg) \] 
is a linearly reparametrized geodesic connecting \(A\) and \(B\) in \(CB(X)\). 
\end{lem}

\begin{proof}
Let \(A,B\in CB(X)\) be given. We note that \(\gamma(t)\in CB(X)\) as it is convex by definition, and since \(\sigma\) is conical, the diameter of \(\gamma(t)\) is bounded by \(\diam(A) + \diam(B)\). 
To prove that \(\gamma\) is a linearly reparameterized geodesic, we will show that 
\[ d_H(\gamma(s),\gamma(t)) \leq |t-s| d_H(A,B) \]
whenever \(s,t\in [0,1]\). To see how this proves the lemma, note that taking \(s=0\) gives \(d_H(A,\gamma(t)) \leq td_H(A,B)\), and taking \(s=1\) gives \(d_H(\gamma(t),B) \leq (1-t) d_H(A,B)\). This gives us 
\begin{align*} 
d_H(A,\gamma(t)) & \leq t d_H(A,B) 
\\ & = d_H(A,B) - (1-t) d_H(A,B) 
\\ & \leq d_H(A,B) - d_H(\gamma(t),B) 
\\ & \leq d_H(A,\gamma(t)) 
\end{align*} 
so \(d_H(A,\gamma(t)) = t d_H(A,B)\). Similarly, \(d_H(\gamma(t),B) = (1-t) d_H(A,B)\). 
Then, assuming without loss of generality \(s<t\), we find 
\begin{align*} 
d_H(\gamma(s),\gamma(t)) & \leq |t-s| d_H(A,B) 
\\ & = td_H(A,B) - sd_H(A,B) 
\\ & = d_H(A,\gamma(t)) - d_H(A,\gamma(s)) 
\\ & \leq d_H(\gamma(t),\gamma(s)) 
\end{align*} 
so \(d_H(\gamma(s),\gamma(t)) = |t-s| d_H(A,B)\).

For simplicity of notation, \(Z_s\) will be the set \(\bigcup_{a\in A} \bigcup_{b\in B} \sigma(a,b,s)\), so that \(\gamma(s) = \cco(Z_s)\). Let \(s,t\in [0,1]\) be such that \(s<t\), and let \(z\in Z_s\) be given. 
We fix \(a\in A\) and \(b\in B\) such that \(z = \sigma(a,b,s)\), and define \(z_t = \sigma(a,b,t)\). 
Furthermore, fix \(b_\varepsilon \in B\) such that \(d(a,b_\varepsilon)< d(a,B) + \varepsilon\), and set \(z_t' = \sigma(a,b_\varepsilon,t)\). 
Letting \(\lambda = s/t\) and using the fact that \(\sigma\) is consistent, we have 
\[ z = \sigma(a,b,s) = \sigma(a,z_t,\lambda) . \] 
Then since \(\sigma\) is conical, we find 
\begin{align*} 
d(z,\cco(Z_t)) & \leq d(z , \sigma(z'_t , z_t , \lambda) ) 
\\ & = d( \sigma(a,z_t,\lambda) , \sigma(z'_t , z_t , \lambda) )
\\ & \leq (1-\lambda) d(a,z'_t) 
\\ & = t (1-\lambda) d(a,b_\varepsilon) 
\\ & < (t-s) d(a,B) + (t-s) \varepsilon 
\end{align*} 
Since this is true for any \(z\in Z_s\), we have \(\sup_{z\in Z_s} d(z,\cco(Z_t)) \leq (t-s) d_H(A,B)\). 
Applying the same arguments as in Lemma \ref{lem: convex distance inequality}, we see that 
\[ \sup \big\{ d(z,\cco(Z_t)) : z\in Z_s \big\} = \sup \big\{ d(z,\cco(Z_t)) : z\in \cco(Z_s) \big\} \] 
and so \(\sup_{z\in \cco(Z_s)} d(z,\cco(Z_t)) \leq (t-s) d_H(A,B)\). 
A symmetrical argument finds \(\sup_{z\in \cco(Z_t)} d(z,\cco(Z_s)) \leq (t-s) d_H(A,B)\). Thus, 
\[ d_H(\gamma(s),\gamma(t)) \leq |t-s| d_H(A,B) \] 
as desired. 
\end{proof}

\begin{thm}\label{thm: conical} 
If \(X\) admits a consistent convex bicombing \(\sigma\), then  
\[ \Sigma (A,B,t) = \cco \bigg( \bigcup_{a\in A} \bigcup_{b\in B} \sigma(a,b,t) \bigg) \] 
defines a conical bicombing on \(CB(X)\). 
\end{thm}

\begin{proof} 
By Lemma \ref{lem: geodesic}, we know that \(\Sigma(A,B,\cdot)\) is a geodesic, so it remains only to show that this bicombing is conical. Let \(A,B,C,D\in CB(X)\) and \(t\in [0,1]\) be given, and define 
\[ U_t =  \bigcup_{a\in A} \bigcup_{b\in B} \sigma(a,b,t) \quad \text{and} \quad V_t =  \bigcup_{c\in C} \bigcup_{d\in D} \sigma(c,d,t) \] 
so that \(\Sigma (A,B,t) = \cco(U_t)\) and \(\Sigma(C,D,t) = \cco(V_t)\). If we can show that 
\[d_H(U_t,V_t) \leq (1- t) d_H(A,C) + t d_H(B,D) \] 
then Lemma \ref{lem: convex distance inequality} will complete the result. 

Let \(u\in U_t\) and \(\varepsilon>0\) be given. Fix \(a_u \in A\) and \(b_u\in B\) such that \(u = \sigma(a_u , b_u , t)\). Next, fix \(c_\varepsilon \in C\) and \(d_\varepsilon \in D\) such that 
\[ d(a_u, c_\varepsilon) < d(a_u , C) + \varepsilon \quad \text{and} \quad d(b_u , d_\varepsilon) < d(b_u , D) + \varepsilon . \] 
Finally, let \(v = \sigma(c_\varepsilon , d_\varepsilon , t ) \in V\). 
Using the fact that \(\sigma\) is conical, we find 
\begin{align*} 
d(u,V_t) & \leq d(u,v) 
\\ & \leq (1-t) d(a_u , c_\varepsilon) + t d(b_u , d_\varepsilon) 
\\ & < (1-t) d(a_u , C) + t d(b_u , D) + \varepsilon . 
\end{align*} 
Given that the above holds for any \(u\in U_t\) and \(\varepsilon >0\), we have 
\begin{align*} 
\sup_{u\in U_t} d(u,V_t) & \leq (1-t) \sup_{a\in A} d(a , C) + t \sup_{b\in B} d(b , D) 
\\ & \leq (1-t) d_H(A,C) + t d_H(B,D) . 
\end{align*} 
A symmetrical argument, starting with \(v\in V_t\), yields 
\[ \sup_{v\in V_t} d(v,U_t) \leq (1-t) d_H(A,C) + t d_H(B,D) \] 
and so \(d_H(U_t,V_t) \leq (1- t) d_H(A,C) + t d_H(B,D)\). 
\end{proof}

This paper does not consider any true applications of Theorem \ref{thm: conical}, but we observe the following immediate corollary.

\begin{cor}
If \(X\) admits a consistent convex bicombing, then \(CB(X)\) is contractible. 
\end{cor} 

\begin{proof} 
Fix any basepoint \(P\in CB(X)\). Since \(CB(X)\) admits a conical bicombing, which we will call \(\Sigma\), one can check that the map \(\phi: CB(X) \times [0,1] \to CB(X)\) defined by 
\[ \phi(A,t) = \Sigma(A,P,t) \] 
is a deformation retraction of \(CB(X)\) to the point \(P\). 
\end{proof}

\section{Special Cases of Theorem \ref{thm: conical}}\label{sec: cases} 

While Theorem \ref{thm: conical} shows that passing to \(CB(X)\) allows us to construct a bicombing which is at least conical, it leaves the question whether or not this bicombing is consistent, or even convex. Part of the difficulty comes from the still somewhat unpredictable nature of convex hulls, even in \(\CAT(0)\) spaces. However, we will show in two special cases -- convex sets of linear spaces and tree graphs -- consistency does hold. 

Before verifying these special cases, we recall one more condition on bicombings, which is easily seen to be preserved in the cases we consider. 

\begin{defn} 
Let \(\sigma\) be a geodesic bicombing on a metric space \(X\). We say that \(\sigma\) is \emph{reversible} if \(\sigma(x,y,t) = \sigma(y,x,1-t)\) for all \(x,y\in X\) and \(t\in [0,1]\). 
\end{defn}

The following corollary is immediate. 

\begin{cor}\label{cor: reversible} 
Assuming the conditions of Theorem \ref{thm: conical}, if \(\sigma\) is additionally reversible, then so is the bicombing \(\Sigma\). 
\end{cor}

We now consider two particular cases where the bicombing given in Theorem \ref{thm: conical} is consistent (and also reversible).

\begin{prop}\label{prop: hrmndr2} 
If \(X\) is a convex subset of a (real or complex) normed space, then 
\[ \Sigma (A,B,t) = \overline{(1-t)A + tB} \] 
defines a reversible consistent convex bicombing on \((CB(X), d_H)\). 
\end{prop} 

\begin{proof} 
Taking \(\sigma\) to be the standard linear bicombing on \(X\) defined in (\ref{eq: normed bicombing}), and recalling that \((1-t)A+tB\) is convex whenever \(A,B\in CB(X)\), it is easy to check that 
\[ \Sigma (A,B,t) = \overline{(1-t)A + tB} = \cco \bigg( \bigcup_{a\in A} \bigcup_{b\in B} \sigma(a,b,t) \bigg) . \]
By Theorem \ref{thm: conical} and Corollary \ref{cor: reversible}, \(\Sigma\) is a reversible conical bicombing on \(CB(X)\). Noting that every consistent conical bicombing is convex, it remains only to show that this bicombing is consistent. 

Let \(A,B\in X\) and \(r,s,t\in [0,1]\) be given such that \(r<s\). Setting \(R = \Sigma(A,B,r)\) and \(S = \Sigma(A,B,s)\), we see that 
\begin{align*} 
\Sigma (R,S,t) & = \overline{(1-t)R + tS} 
\\ & = \overline{(1-t)\left(\overline{(1-r)A+rB}\right) + t\left(\overline{(1-s)A+sB}\right)} 
\\ & = \overline{\overline{(1-t)(1-r)A+(1-t)rB} + \overline{t(1-s)A+tsB}} 
\\ & = \overline{(1-((1-t)r + ts))A + ((1-t)r + ts)B} 
\\ & = \Sigma(A,B,(1-t)r + ts) 
\end{align*} 
so \(\Sigma\) is consistent. 
\end{proof}

Our second example is for \(\R\)-trees, which generalize metric tree graphs. See also \cite[Section II.1]{MS1984} or \cite[Example II.1.15]{bridson}. 

\begin{defn} 
An \emph{\(\R\)-tree} is a metric space \(T\) such that there is a unique geodesic joining any pair of points in \(T\) and -- letting \([x,y]\) denote the image of the geodesic starting at \(x\) and ending at \(y\) -- if \([x,y] \cap [y,z] = \{y\}\), then \([x,y]\cup[y,z] = [x,z]\). 
\end{defn}

\begin{prop} 
If \(T\) is an \(\R\)-tree, then the bicombing \(\Sigma\) given in Theorem \ref{thm: conical} is a reversible consistent convex bicombing on \(CB(T)\). 
\end{prop} 

\begin{proof} 
Since every \(\R\)-tree is a \(\CAT(0)\) space \cite[Example II.1.15]{bridson}, \(T\) admits a unique reversible consistent convex bicombing \(\sigma\). Theorem \ref{thm: conical} and Corollary \ref{cor: reversible} immediately show that \(\Sigma\) is a reversible conical bicombing on \(CB(T)\), so it remains only to verify that \(\Sigma\) is consistent. 

First, we note that any subset of a tree graph is convex if and only if it is connected (a set which is connected but not convex contradicts the property that \(T\) is a tree). Furthermore, the continuity of \(\sigma\) implies that 
\[ \bigcup_{a\in A} \bigcup_{b\in B} \sigma(a,b,t) \] 
is connected whenever \(A\) and \(B\) are. Therefore, taking the convex hull in the definition of \(\Sigma\) is unnecessary, and we have 
\[ \Sigma(A,B,t) = \overline{\bigcup_{a\in A} \bigcup_{b\in B} \sigma(a,b,t)} \] 
for any \(A,B\in CB(T)\). The fact that \(\Sigma\) is consistent now follows immediately from our hypothesis that \(\sigma\) is consistent. 
\end{proof}


To end this section, we make a final observation concerning the geometry of \(CB(X)\). One might wonder if \(CB(X)\) could display nonpositive curvature in the sense of a \(\CAT(0)\) or Busemann space. However, as the following example shows, \(CB(X)\) is not necessarily uniquely geodesic, even if \(X\) is.

\begin{ex} 
Consider \(CK(\R)\) and let \(A=[-1,1]\), \(B =[-2,3]\), \(U=[-1,2]\), and \(V = [-2,2]\). Then \(d_H(A,B) = 2\) and 
\[ d_H(A,U) = d_H(A,V) = d_H(B,U) = d_H(B,V) = 1 \] 
so there are at least two shortest paths connecting \(A\) and \(B\). 
This example can be adapted to \(CK(\R^n)\) by defining \(A = [-1,1]\times [0,1]^{n-1}\), \(B = [-2,3] \times [0,1]^{n-1}\), etc. 
\end{ex}

On the other hand, it should be pointed out that introducing a different metric besides the Hausdorff metric can potentially improve the geometry. For example, \cite{ibragimov} achieves some success by establishing a metric on the nonsingleton elements of \(CB(X)\), which is Gromov hyperbolic and views the base space \(X\) as the boundary at infinity.

\section{Hyperspaces of Proper Geodesic Spaces}\label{sec: bicombing k(x)} 

We recall the main theorem of \cite{sosov}. 

\begin{thm}[Sosov, \cite{sosov}] 
The metric space \((B(X),d_H)\) has an intrinsic metric if and only if \((X,d)\) has an intrinsic metric. 
\end{thm}

By virtue of the generalized Hopf--Rinow theorem\footnote{Every complete and locally compact space with intrinsic metric is a proper geodesic space.} and the hyperspace inheritance properties of \cite{michael}, we have the following corollary. 

\begin{cor}\label{cor: hausdorff geodesic}
Let \((X,d)\) be a proper metric space. The space \((K(X) , d_H)\) is a geodesic space if and only if \((X,d)\) is a geodesic space. 
\end{cor}

This leads us to ask: Given a geodesic bicombing on a proper metric space \(X\), can we define a bicombing on \(K(X)\)? Moreover, can we find a bicombing which satisfies any of the additional properties given in Definition \ref{defn: bicombing}? 
We answer the former question by constructing a bicombing on \(K(X)\) in Proposition \ref{prop: compact set bicombing}, but leave the latter open, as our bicombing is not necessarily conical (Example \ref{ex: not conical}). 
Before proving the proposition, we first observe that the same bicombing from the previous section will not work in this case. 

\begin{ex}\label{ex: failed bicombing} 
Consider the interval \([0,1]\) equipped with the standard absolute value metric and the consistent convex bicombing from (\ref{eq: normed bicombing}). Let \(A,B\in K([0,1])\) be the sets \(A = \{0,1\}\) and \(B = \{0.3 , 0.4\}\). 
Now define 
\[ M = \bigcup_{a\in A} \bigcup_{b\in B} \sigma(a,b,1/2) = \{ 0.15 , 0.2 , 0.65, 0.7\} . \] 
It is straightforward to compute 
\[ d_H(A,B) = 0.6 \quad d_H(A,M) = 0.35 \quad \text{and} \quad d_H(B,M) = 0.3 . \] 
Therefore, \(M\) is not a midpoint of \(A\) and \(B\), and so \(\bigcup_{a\in A} \bigcup_{b\in B} \sigma(a,b,t)\) does not define a geodesic connecting \(A\) and \(B\) (even if we take the closed convex hull). 
\end{ex}

However, if we introduce a couple more tools, then we can construct a geodesic bicombing on \(K(X)\) whenever \(X\) is a proper geodesic space. Define \(P_A\) to be the \emph{metric projection} of \(x\) onto \(A\), 
\[ P_A(x) = \{a\in A : d(x,a) = d(x,A)\} . \] 
That is, \(P_A(x)\) is the set of all points in \(A\) which realize the distance from \(x\) to \(A\). If \(A\) is compact, then certainly \(P_A(x)\) is nonempty for every \(x \in X\). We also define the multifunction \(\omega\) (similar to that in \cite{sosov}) by  
\[ \omega(a,b,t) = \{x \in X : d(a,x) = t d(a,b) \text{ and } d(x,b) = (1-t) d(a,b) \} . \] 
We can think of \(\omega(a,b,t)\) as the collection of \(\sigma(a,b,t)\) over all possible bicombings on \(X\).

\begin{prop}\label{prop: compact set bicombing} 
Let \((X,d)\) be a proper metric space which admits a geodesic bicombing. Then the map 
\[ \Sigma (A,B,t) = \bigg( \bigcup_{a\in A} \bigcup_{u\in P_B(a)} \omega(a,u,t) \bigg) \cup \bigg( \bigcup_{b\in B} \bigcup_{v\in P_A(b)} \omega(v,b,t) \bigg) \] 
defines a geodesic bicombing on \((K(X),d_H)\). 
\end{prop}

\begin{proof} 
First, we need to verify that \(\Sigma(A,B,t)\) is in fact compact. Let \(A,B\in K(X)\) and \(t\in [0,1]\) be given and define 
\[ Z_t = \bigcup_{a\in A} \bigcup_{u\in P_B(a)} \omega(a,u,t) . \] 
If we can show that \(Z_t\) is closed and bounded, then it is compact (as \(X\) is proper); and therefore \(\Sigma(A,B,t)\) is compact, as it is the union of two compact sets. 


To verify that \(Z_t\) is closed, let \(\{z_n\}_n\) be a sequence in \(Z_t\) such that \(z_n\to x\) for some \(x\in X\). 
By the definition of \(Z_t\), there is a sequence \(\{a_n\}_n\subseteq A\) such that \(z_n \in \omega(a_n,u_n,t)\) for each \(n\) (where \(u_n\in P_A(a_n)\)). 
Since \(A\) is compact, \(\{a_n\}_n\) contains a convergent subsequence, say \(\{a_{n_k}\}_k\), which converges to \(a\in A\). 
Furthermore, since \(B\) is compact, \(\{u_{n_k}\}_k\) contains a convergent subsequence, say \(\{u_{n_{k_j}}\}_j\), converging to some \(u\in B\). 
For simplicity, we substitute the index \(n = n_{k_j}\) so that 
\[ z_n \to x , \quad a_n \to a , \quad \text{and } u_n \to u .\] 
Noting that \(d(a_{n},u_{n}) = d(a_{n},B)\) for all \(n\), we have \(d(a,u) = d(a,B)\), so \(u\in P_B(a)\). 
By continuity of the distance function, \(d(a_n , z_n) \to d(a,x)\) and \(d(z_n, u_n) \to d(x,u)\). 
Given that \(z_n \in \omega(a_n,u_n,t)\), we find \(d(a,x) = td(a,u)\) and \(d(x,u) = (1-t)d(a,u)\), so 
\[ x\in \omega(a,u,t) \subseteq Z_t . \] 
Thus, \(Z_t\) is closed. 


To show that \(Z_t\) is bounded, let \(z,z'\in Z_t\) be given and fix \(a,a'\in A\) such that \(z\in \omega(a,u,t)\) and \(z'\in \omega(a',u',t)\) (where \(u\in P_B(a)\) and \(u'\in P_B(a')\)). Applying the triangle inequality, and the definition of \(\omega\), we find 
\begin{align*} 
d(z,z') & \leq d(z,a) + d(z',a') + d(a,a') 
\\ & = td(a,B) + td(a',B) + d(a,a') 
\\ & \leq 2td_H(A,B) + \text{diam}(A) . 
\end{align*} 
Since \(A\) and \(B\) are compact, \(2td_H(A,B) + \text{diam}(A)\) is finite, so \(Z_t\) is bounded. We conclude that \(\Sigma(A,B,t)\) is compact.

Finally, we employ the same technique used in Lemma \ref{lem: geodesic} to check that \(\Sigma(A,B,\cdot)\) is a linearly reparametrized geodesic.

Let \(s,t\in [0,1]\) and \(z\in \Sigma(A,B,s)\) be given. Either 
(i) there are points \(a\in A\) and \(u\in P_B(a)\) such that \(z\in \omega(a,u,s)\); or 
(ii) there are points \(b\in B\) and \(v\in P_A(b)\) such that \(z\in \omega(v,b,t)\). 
If case (i), then let \(\gamma: [0,1]\to X\) be a linearly reparametrized geodesic such that \(\gamma(0) = a\), \(\gamma(1) = u\), and \(\gamma(s) = z\). Setting \(z_t = \gamma(t) \in \Sigma(A,B,t)\), 
\[ d(z,\Sigma(A,B,t)) \leq d(z,z_t) = |t-s| d(a,u) = |t-s| d(a,B) \leq |t-s| d_H(A,B) . \] 
If case (ii), then let \(\eta: [0,1]\to X\) be a linearly reparametrized geodesic such that \(\eta(0) = v\), \(\eta(1) = b\), and \(\eta(s) = z\). Setting \(z_t = \eta(t) \in \Sigma(A,B,t)\), 
\[ d(z,\Sigma(A,B,t)) \leq d(z,z_t) = |t-s| d(v,b) = |t-s| d(b,A) \leq |t-s| d_H(A,B) . \] 
In either case, \(d(z,\Sigma(A,B,t)) \leq |t-s|d_H(A,B)\), so 
\[ \sup_{z\in \Sigma(A,B,s)} d(z,\Sigma(A,B,t)) \leq |t-s| d_H(A,B) . \] 
As usual, a symmetrical argument yields 
\[ \sup_{z\in \Sigma(A,B,t)} d(z,\Sigma(A,B,s)) \leq |t-s| d_H(A,B) \] 
so we conclude that \(d_H(\Sigma(A,B,s) , \Sigma(A,B,t) ) \leq |t-s| d_H(A,B)\). 
\end{proof}

Unfortunately, the bicombing given in Proposition \ref{prop: compact set bicombing} is not generally conical, even if the bicombing of the base space is consistent and convex, as the following example shows. 

\begin{ex}\label{ex: not conical}
As in Example \ref{ex: failed bicombing}, we consider the interval \([0,1]\) equipped with the standard absolute value metric and the consistent convex bicombing from (\ref{eq: normed bicombing}). Let \(A,B,C\in K([0,1])\) be the sets 
\[ A = \{0,1\} \quad B = \{0.3 , 0.4 \} \quad C = \{0.1 , 0.6\} . \] 
Using the bicombing \(\Sigma\) from Proposition \ref{prop: compact set bicombing}, we find 
\[ \Sigma(A,B,1/2) = \{ 0.15 , 0.2 , 0.7\} \quad \text{and} \quad \Sigma(A,C,1/2) = \{0.05 , 0.8\} . \] 
Furthermore, 
\[ d_H\big( \Sigma(A,B,1/2) , \Sigma(A,C,1/2) \big) = 0.15 > 0.1 = \frac{1}{2} d_H\big( B, C \big) \] 
so the bicombing of Proposition \ref{prop: compact set bicombing} is not conical in this case. 
\end{ex}

\begin{ack}
The author would like to thank Peter Oberly, Joel H. Shapiro, and J.J.P. Veerman for many helpful comments. 
\end{ack}

\bibliography{bicombing} 
\bibliographystyle{plain}

\end{document}